\documentclass{amsart}
\usepackage{amssymb}
\usepackage{amsthm}
\newtheorem{theor}{Theorem} 

\theoremstyle{definition} \newtheorem{defin}{Definition}
\newtheorem{ex}{Example}
\theoremstyle{remark} \newtheorem{rem}{Remark}
\newcommand{\w}{\widehat}
\begin{document}

\title{Homologies of moduli space $\mathcal{M}_{2,1}$}
\author{Yury Kochetkov}
\date{}
\email{yuyuk@prov.ru} \maketitle

\begin{abstract}{We consider the space $\mathcal{M}_{2,1}$ --- the open moduli
space of complex curves of genus 2 with one marked point. Using language of
chord diagrams we describe the cell structure of $\mathcal{M}_{2,1}$ and cell
adjacency. This allows us to construct matrices of boundary operators and
compute Betty numbers of $\mathcal{M}_{2,1}$ over $\mathbb{Q}$.}\end{abstract}

\section{Introduction}

A graph $G$ is correctly embedded into a compact topological oriented surface
$S$, if its complement is homeomorphic to a 2-dimensional disk. If a graph $G$
is correctly imbedded into $S$ and lengths of its edges are given, then the
Jenkins-Schtrebel construction allows one to uniquely define a complex
structure on $S$ (see \cite{Kontz}). In what follows we will assume that the
sum of lengths of edges is 1.

Let $\mathcal{M}_{2,1}$ be the open moduli space of genus 2 complex curves with
one marked point. A topological cell complex $\mathcal{M}_{2,1}^{\rm comb}$,
which is homeomorphic to the space $\mathcal{M}_{2,1}$, has the following
combinatorial description (see \cite{Lando}, Chapter 4, for example). Let us
consider the set of all pairwise nonisomorphic graphs, correctly embedded into
$S_2$ --- topological compact genus 2 surface. Isomorphism here is the
isomorphism of embedded graphs, i.e. isomorphism must preserve the cyclical
order of edges under the counterclockwise going around of each vertex.

To a graph $\Gamma$ with $s$ edges, correctly imbedded into $S_2$, we
correspond the simplex $\Delta_\Gamma$, which is isometric to standard simplex
$$\{x_1,\ldots,x_s\in\mathbb{R}^s\,|\,x_1+\cdots+x_s=1,\quad
x_1,\ldots,x_s>0\}.$$ Here numbers $x_1,\ldots,x_s$ are lengths of edges of
$\Gamma$. Cells of the highest dimension 8 correspond to correctly imbedded
3-valent graphs (a graph is called 3-valent, if each its vertex has degree 3).
A 3-valent graph correctly imbedded in to $S_2$ has 6 vertices and 9 edges.

If a correctly imbedded graph $\Gamma'$ is obtained by a contraction of some
edge of $\Gamma$, then to $\Gamma'$ some face of $\Delta_\Gamma$ is
corresponded.

A correctly imbedded graph $\Gamma$ may have a nontrivial group of
automorphisms. We can define an action of this group on the simplex
$\Delta_\Gamma$ also. The cell of the space $\mathcal{M}_{2,1}^{\rm comb}$ is
either the simplex $\Delta_\Gamma$, if the corresponding graph doesn't have any
nontrivial automorphisms, or the factor of $\Delta_\Gamma$ by the action of the
group.

Thus defined space $\mathcal{M}_{2,1}^{\rm comb}$ is \emph{noncompact}. In this
work we will find its homologies over $\mathbb{Q}$.

\section{Graphs, glueings and diagrams}

The topological surface $S_2$ with correctly imbedded 3-valent graph can be
realized as a glueing of 18-gon. Here the embedding correctness is
automatically fulfilled and the 3-valency and the genus 2 condition give 9
pairwise nonisomorphic glueings (two glueings of $2n$-gon are isomorphic, if
some rotation of $2n$-gon transforms one onto another). The group of
automorphisms of a graph is a cyclic group of automorphisms of the
corresponding diagram of glueing. Combinatorics of cells of the highest
dimension is described in \cite{Koch}.

A contraction of an edge of graph corresponds to a contraction of the pair of
identified sides of corresponding glueing. Thus cells of dimension 7 correspond
to glueings of 16-gons and cells of dimension 6 --- to glueings of 14-gons.

We will use schemas of chord diagrams for enumeration of glueings.

\begin{ex} Gaussian word $a_1a_2a_3a_1^{-1}a_4a_2^{-1}a_4^{-1}a_3^{-1}$
that describes the glueing of octagon
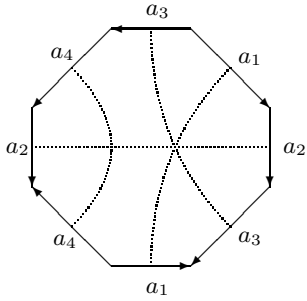
\[\begin{picture}(350,120) \put(15,75){\vector(0,-1){30}}
\put(45,15){\vector(-1,1){30}} \put(45,105){\vector(-1,-1){30}}
\put(45,15){\vector(1,0){30}} \put(75,105){\vector(-1,0){30}}
\put(75,105){\vector(1,-1){30}} \put(105,75){\vector(0,-1){30}}
\put(105,45){\vector(-1,-1){30}} \qbezier[60](15,60)(60,60)(105,60)
\qbezier[60](60,15)(60,60)(90,90) \qbezier[60](60,105)(60,60)(90,30)
\qbezier[50](30,90)(60,60)(30,30) \put(110,58){\small $a_2$} \put(5,58){\small
$a_2$} \put(93,92){\small $a_1$} \put(58,5){\small $a_1$} \put(93,24){\small
$a_3$} \put(58,110){\small $a_3$} \put(23,23){\small $a_4$} \put(22,94){\small
$a_4$}

\put(150,70){\small we will write Gaussian}  \put(150,60){\small word as string
12314243,} \put(150,50){\small and represent the glueing } \put(150,40){\small
as a chord diagram:}

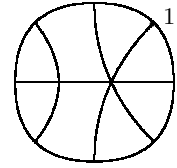
\qbezier(290,60)(290,90)(320,90) \qbezier(320,90)(350,90)(350,60)
\qbezier(350,60)(350,30)(320,30) \qbezier(320,30)(290,30)(290,60)
\put(342,82){\circle*{2}} \qbezier(342,82)(320,60)(320,30)
\put(290,60){\line(1,0){60}} \qbezier(342,38)(320,60)(320,90)
\qbezier(298,38)(315,60)(298,82) \put(346,82){\small 1}
\end{picture}\] simplifying it to a scheme
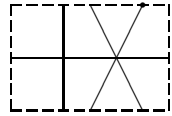
\[\begin{picture}(60,50) \multiput(0,5)(6,0){10}{\line(1,0){4}}
\multiput(0,45)(6,0){10}{\line(1,0){4}} \multiput(0,5)(0,6){7}{\line(0,1){4}}
\multiput(60,5)(0,6){7}{\line(0,1){4}} \put(0,25){\line(1,0){60}}
\put(20,5){\line(0,1){40}} \put(30,5){\line(1,2){20}}
\put(30,45){\line(1,-2){20}} \put(50,45){\circle*{2}} \end{picture}\] The point
at diagram denotes the beginning of clockwise numeration of chords. \end{ex}

All glueings we will work with are enumerated in Appendices 1 -- 6. A glueing
will be denoted $\gamma(k)_l$, where $k$ is the number of edges of
corresponding graph and $l$ is the number of this glueing in the list of
glueings of $2k$-gons. For each glueing we give its chord diagram, Gaussian
word and group of automorphisms (if it is nontrivial). Gaussian word defines
the numeration of chords, which will be called \emph{standard}.

All chord schemas, enumerated in Appendices, define a genus 2 curves with
correctly embedded graphs. Each diagram $\gamma(k)_i$, $k=8,7,6,5,4$ is
obtained from some scheme $\gamma(k+1)_j$ by deletion of a chord.

A scheme $\gamma(k)_i$ defines a $(k-1)$-dimensional cell in the space
$\mathcal{M}_{2,1}^{\rm comb}$. Cells that correspond to genus 2 schemas,
obtained from $\gamma(k)_i$ by deletion of a chord, constitute the
$(k-2)$-dimensional boundary of our cell.

A glueing of $2n$-gon (a chord diagram of glueing) will be called symmetric if
it has a nontrivial group of automorphisms. This group is a cyclic group of
rotations of $2n$-gon. As sides of $2n$-gon (chords of diagram) are numerated,
then generator of group of automorphisms defines a permutation from $S_n$. We
will call a glueing (scheme) \emph{even-symmetric}, if this permutation is
even, and \emph{odd-symmetric} in the opposite case.

A cell of the space $\mathcal{M}_{2,1}^{\rm comb}$ will be called
\emph{simple}, if this cell is a simplex. A cell=factorized simplex will be
called \emph{special even}, if the corresponding glueing is even symmetric, and
\emph{special odd}, if the corresponding glueing is odd symmetric.

The topological space $\mathcal{M}_{2,1}^{\rm comb}$ consists of (see
Appendices 1--6):
\begin{itemize}
    \item 9 cells of dimension 8: 3 are simple, 5 are simplices, factorized
    by action of the group $\mathbb{Z}_2$, and one special cell is a simplex
    factorized  by action of the group $\mathbb{Z}_3$ (here all special cells
    are even);
    \item 29 cells of dimension 7: 24 cells are simple, 4 are simplices,
    factorized by action of the group $\mathbb{Z}_2$ (these cells are even),
    and one special cell is a simplex factorized by action of the group
    $\mathbb{Z}_4$ (this cell --- $\delta(7)_{20}$ is odd);
    \item 52 cells of dimension 6: 41 cells are simple, 11 are special cells
     --- simplices, factorized by action of the group $\mathbb{Z}_2$ (two of
     them --- $\delta(6)_{39}$ and $\delta(6)_{40}$ are even, all others are
     odd);
    \item 45 cells of dimension 5: 37 cells are simple, 5 are simplices
    factorized by action of the group $\mathbb{Z}_2$, one is a simplex
    factorized by action of the group $\mathbb{Z}_3$, one is a simplex
    factorized by action of the group $\mathbb{Z}_4$ and one is a simplex
    factorized by action of the group $\mathbb{Z}_6$ (two special cells ---
    $\delta(5)_{39}$ and $\delta(5)_{41}$ are even, all others are
    odd);
    \item  21 cells of dimension 4: 14 cells are simple, 5 are simplices
    factorized by action of the group $\mathbb{Z}_2$, one is a simplex
    factorized by action of the group $\mathbb{Z}_5$ and one is a simplex
    factorized by action of the group $\mathbb{Z}_{10}$ (one special cell ---
    $\delta(4)_{19}$ is odd, all others are even);
    \item 4 cells of dimension 3: two are simple, one is a simplex
    factorized by action of the group $\mathbb{Z}_2$ and one is a simplex
    factorized by action of the group $\mathbb{Z}_8$ (the cell
    $\delta(4)_{1}$ is special even and the cell $\delta(4)_4$ is special odd).
\end{itemize}

Euler characteristic of open moduli spaces $\mathcal{M}_{2,1}$ is
9-29+52-45+21-4=4. And its "orbifoldic" Euler characteristic is
\begin{multline*}\left(3+\frac 52+\frac 13\right)-\left(24+\frac
42+\frac 14\right)+ \left(41+\frac{11}{2}\right)-\left(37+\frac 52+\frac
13+\frac 14+\frac 16\right)+\\+\left(14+\frac 52+\frac
15+\frac{1}{10}\right)-\left(2+\frac 12+\frac 18\right)=\frac{1}{120}.
\end{multline*} This result is in agreement with Harer-Zagier formula \cite{Harer}.

\section{Homologies}

At first we'll explain the notion of induced numeration.

\begin{defin} The scheme, obtained by deletion of the $i$-th chord from a scheme
$\gamma$ will be denoted $\gamma[i]$. We can define a numeration of chords of
$\gamma[i]$ in the following way: if a chord $c$ has number $j<i$ in $\gamma$,
then $c$ has the same number in $\gamma[i]$; if a chord $c$ has number $j>i$ in
$\gamma$, then $c$ has number $j-1$ in $\gamma[i]$. Thus defined numeration of
chords of the scheme $\gamma[i]$ will be called \emph{induced}. We will use
notations $\gamma[i]_{\text{ind}}$ and $\gamma[i]_{\text{st}}$ (i.e. standard)
when necessary.
\end{defin}

If a cell $\delta$ corresponds to a scheme $\gamma$, then by $\delta[i]$ we
will denote the cell that corresponds to scheme $\gamma[i]$.

\begin{rem} As $\gamma[i]=\gamma[p]_q$ for some $p$ and $q$, then scheme
$\gamma[i]$ has the standard numeration also. The standard and induced
numerations are usually different. \end{rem}

\begin{defin} Let $C_k=C_k(\mathcal{M}_{2,1}^{\rm comb},\mathbb{Q})$ be the
space of $k$-dimensional chains with rational coefficients. The boundary
operator $\partial_k:C_k\to C_{k-1}$ maps each cell $\delta(k)_i$ into chain
$\sum_{j=1}^{k+1} \alpha_j \delta(k)_i[j]\in C_{k-1}$. Coefficients $\alpha_j$
are defined by following conditions.
\begin{enumerate}
    \item Let cells $\delta(k)_i$ and $\delta(k)_i[j]$ be simple. The
    identification of schemas \\ $\gamma(k+1)_i[j]_{\text{ind}}$ and
    $\gamma(k+1)_i[j]_{\text{st}}$ defines some renumeration of chords of the scheme
    $\gamma(k+1)_i[j]_{\text{ind}}$, i.e. defines a permutation $\sigma\in S_k$
    of parity $p$, $p=0,1$. Set $\alpha_j:=(-1)^{j+p-1}$.
    \item Let the sell $\delta(k)_i$ be simple and the cell $\delta(k)_i[j]$ be
    special even --- a simplex factorized by action of cyclic group of order
    $m$. The identification of schemas $\gamma(k+1)_i[j]_{\text{st}}$ and
    $\gamma(k+1)_i[j]_{\text{ind}}$ doesn't define the permutation $\sigma\in S_k$
    uniquely, but all such permutation have the same parity $p$. Set
    $\alpha_j:=(-1)^{j+p-1}m$.
    \item  Let the sell $\delta(k)_i$ be simple and the cell $\delta(k)_i[j]$ be
    special odd. Set $\alpha_j:=0$.
    \item Let the cell $\delta(k)_i$ be special even: a cyclic group
    $\mathbb{Z}_r$ acts on the scheme $\gamma(k+1)_i$ (and on the simplex
    $\Delta(k)_i$). The $j$-th chord either belongs to an orbit of
    cardinality $r$, or itself is an orbit. In the first case the deletion of this
    chord gives us a nonsymmetric scheme and the coefficient $\alpha_j$ is
    defined according to the rule (1) above.  In the second case a cyclic group
    $\mathbb{Z}_q$, where $r$ divides $q$, acts on the scheme $\gamma(k+1)_i[j]$
    (and on the corresponding simplex). If this scheme is even symmetric, then
    $\alpha_j:=(-1)^{n+p-1}q/r$ (where parity is computed in above defined way).
    If this scheme is odd symmetric, then $\alpha_j:=0$.
    \item Let the cell $\delta(k)_i$ be special odd and a cyclic group
    $\mathbb{Z}_r$ acts on the scheme $\gamma(k+1)_i$ (and on the simplex
    $\Delta(k)_i$). If the $j$-th chord belongs to an orbit of cardinality $>1$,
    then set $\alpha_j:=0$. If this chord itself is an orbit, then the scheme
    $\gamma(k+1)_i[j]$ is odd symmetric and a cyclic group $\mathbb{Z}_q$, where
    $r$ divides $q$, acts on it. Now set $\alpha_j:=(-1)^{j+p-1}q/r$ (where parity
    is computed in above defined way).
\end{enumerate} \end{defin}

\begin{rem} The union of special odd cells constitute a subcomplex in
$\mathcal{M}_{2,1}^{\rm comb}$. \end{rem}

Matrices of boundary maps are presented in Appendices.

\begin{theor} $\partial_{k-1}\circ\partial_k=0$. \end{theor}

\begin{proof} Let $\gamma$ be a $(k+1)$-scheme. Let us consider chords with
numbers $i$ and $j$. We need to prove that the scheme $\gamma[i][j]$ has zero
coefficient in the sum $\partial^2(\gamma)$. If schemas $\gamma$, $\gamma[i]$,
$\gamma[j]$ and $\gamma[i][j]$ are simple, then this statement is a consequence
of the analogous result for simplicial homologies. It means (if we ignore
conditions 2-5) that the sign of the passage
$\gamma\to\gamma[i]\to\gamma[i][j]$ is opposite to the sign of passage
$\gamma\to\gamma[j]\to\gamma[j][i]$. Thus, it remains to take into account
symmetry conditions. Let us consider two typical cases.
\par\medskip\noindent
Schemas $\gamma$, $\gamma[i]$ and $\gamma[i][j]$ are simple and scheme
$\gamma[j]$ is odd symmetric. Let $\varphi$ be a generator of symmetry group of
the scheme $\gamma[j]$ and $i=i_1,\ldots,i_m$ be the orbit of $i$-th chord
(here $m$ is even and $i_2=\varphi(i_1)$, $i_3=\varphi^2(i_1),\ldots$).
Symmetry $\varphi$ defines a new numeration on $\gamma[j]$ such, that $i_1$-th
chord now has number $i_2$. New induced numeration on $\gamma[i][j]$ after
deletion of $i_2$-th (former $i_1$-th) chord is the same as old induced
numeration on $\gamma[j][i_2]$. It remains to note that in both cases we
multiply by $(-1)^{i_2-1}$ and that permutation defined by $\varphi$ is odd.
\par\medskip\noindent
Schemas $\gamma$ and $\gamma[i]$ are simple, scheme $\gamma[j]$ is odd
symmetric and scheme $\gamma[i][j]$ is even symmetric. As above, let $\varphi$
be a generator of symmetry group of the scheme $\gamma[j]$ and
$i=i_1,\ldots,i_m$ ($m$ is even) be the orbit of $i$-th chord. Schemas
$\gamma[i_1][j],\ldots,\gamma[i_m][j]$ all are some scheme $\gamma(k)_s$. Let
$p_k$ be the parity of renumeration from $\gamma[i_k][j]_{\text{st}}$ to
$\gamma[i_k][j]_{\text{ind}}$. Now it is clear that $m/2$ of these $p_k$ are
zeroes and $m/2$ are units. \end{proof}

\par\medskip
Theorem 2 is the main result of the work.

\begin{theor} \emph{Let $M_i$ be the matrix of the boundary map
$\partial_i$, then}
$${\rm rk}(M_8)=8,\,{\rm rk}(M_7)=20,\,{\rm rk}(M_6)=27,\,{\rm rk}(M_5)=17,\, {\rm
rk}(M_4)=3.$$ \emph{Thus, Betty numbers are:}
$$b_8=1,\, b_7=1,\, b_6=5,\, b_5=1,\, b_4=1,\, b_3=1.$$ \end{theor}

\section{Appendix 1. Schemas of nine-edge glueings.}

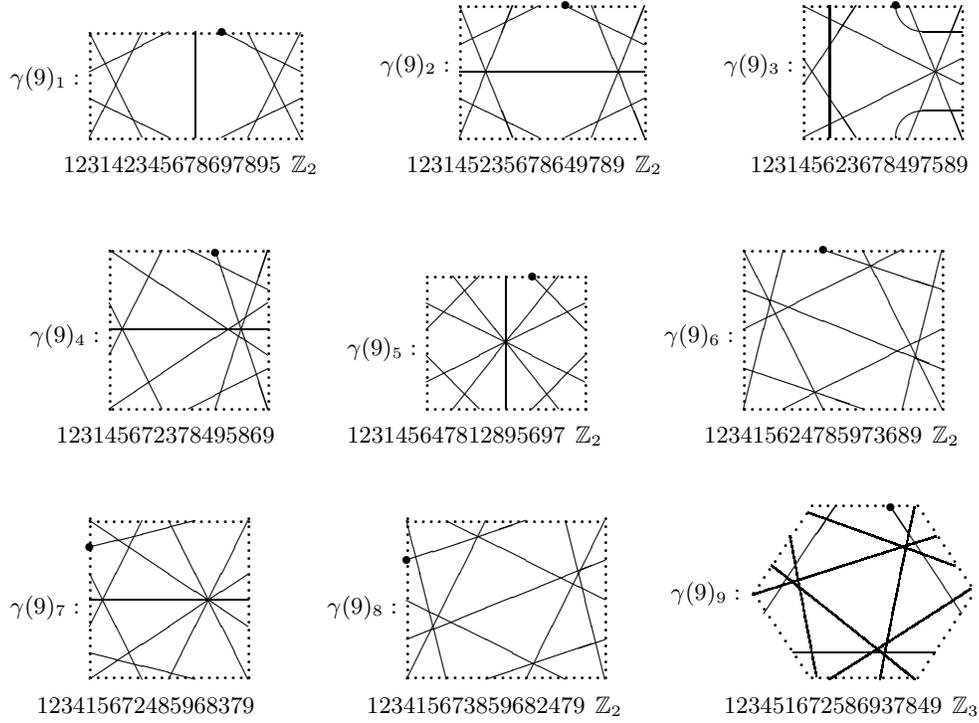
\[
\]

\section{Appendix 2. Schemas of eight-edge glueings.}

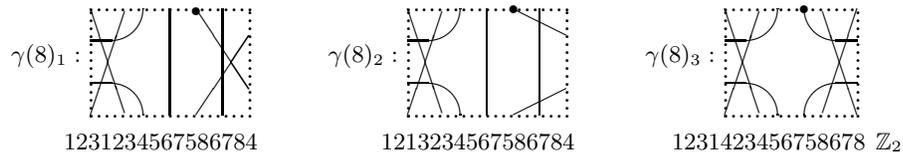
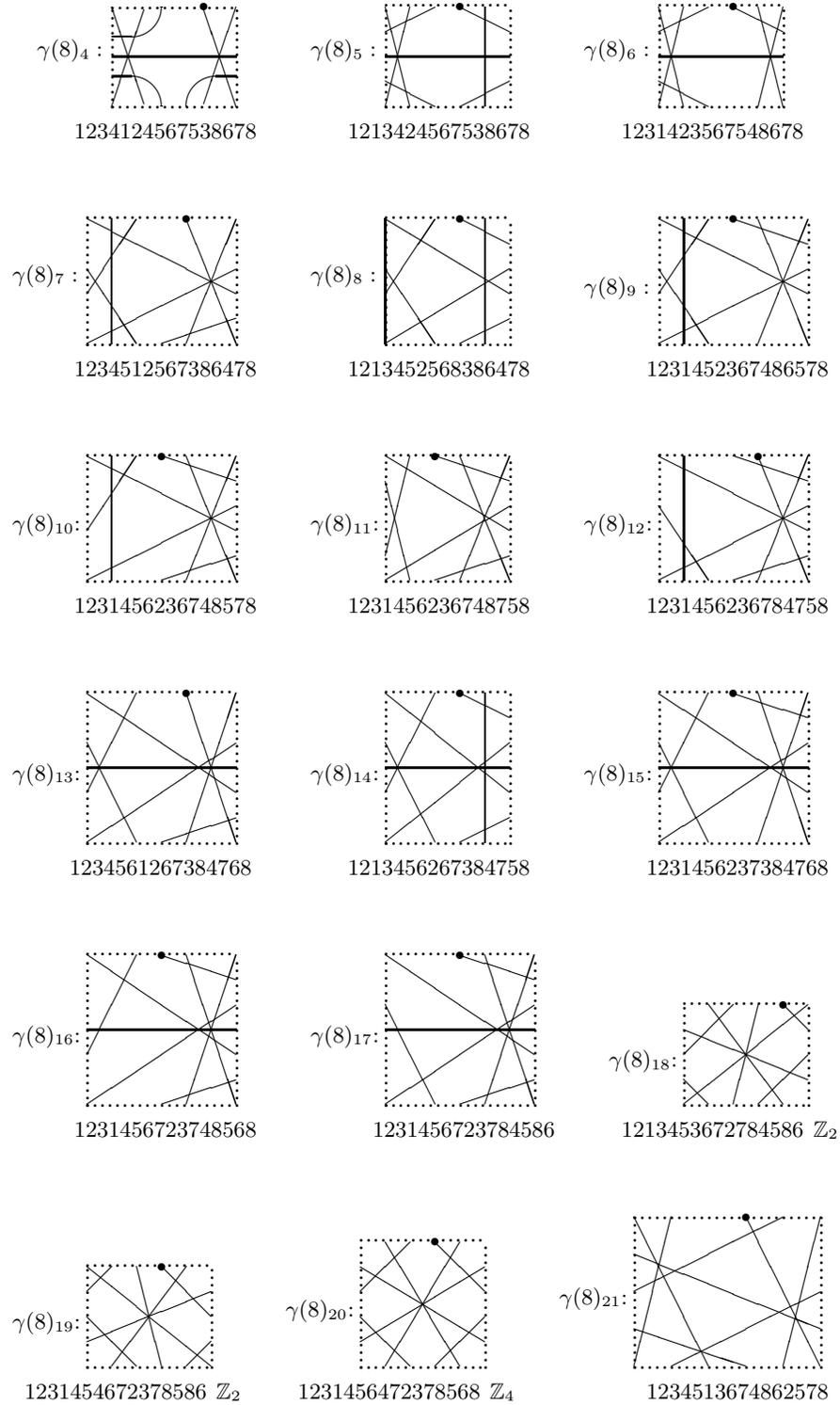
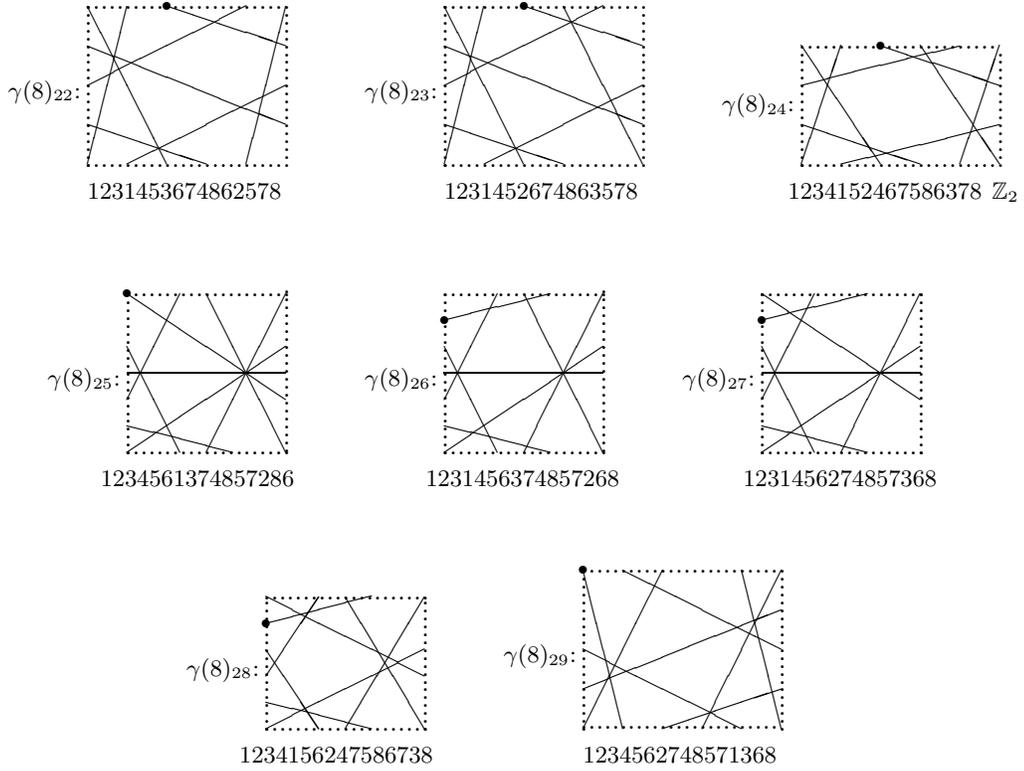
\[%
\]

\section{Appendix 3. Schemas of seven-edge glueings.}

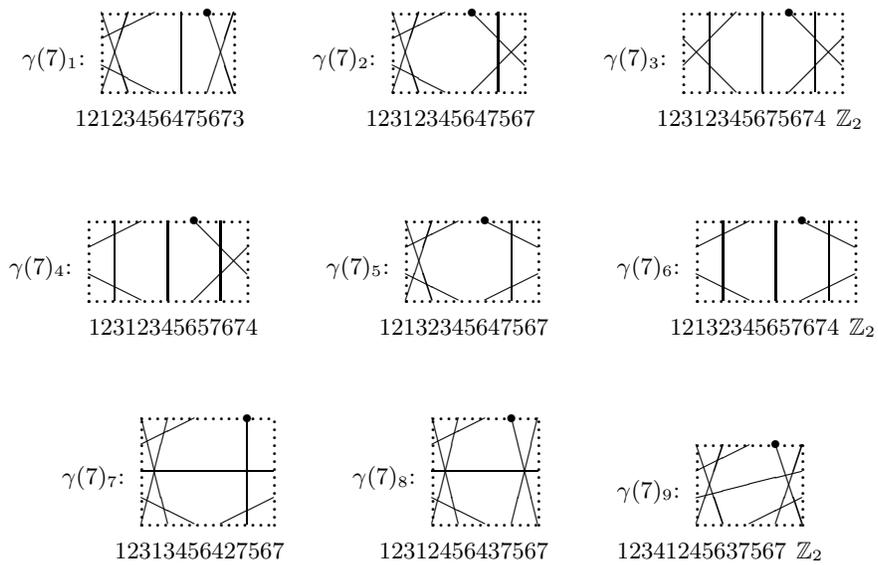
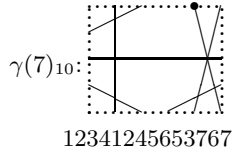
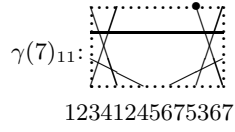
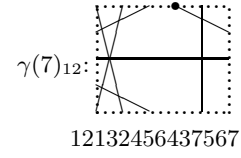
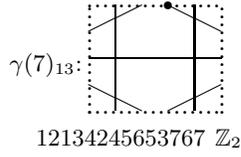
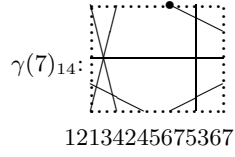
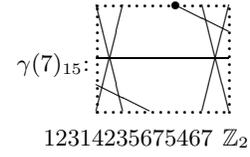
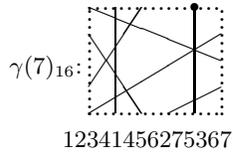
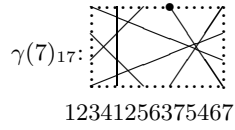
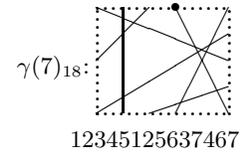
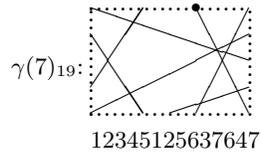
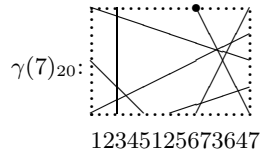
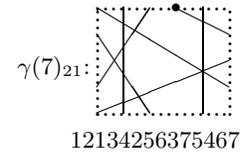
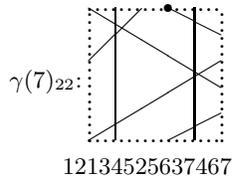
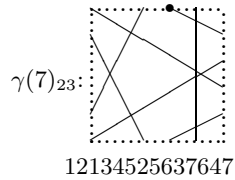
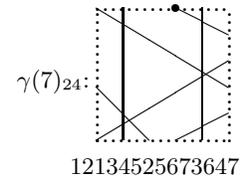
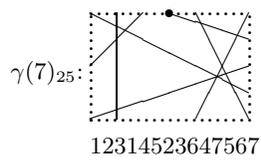
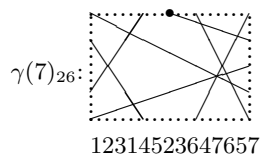
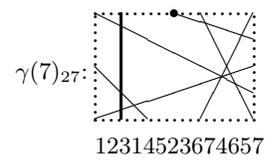
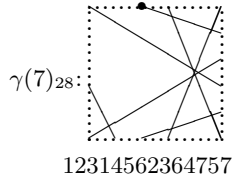
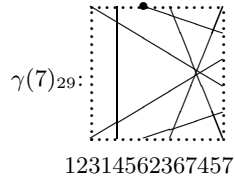
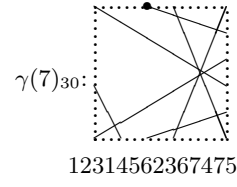
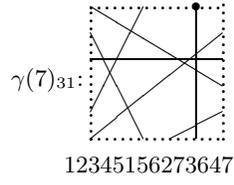
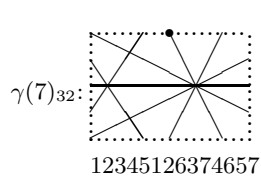
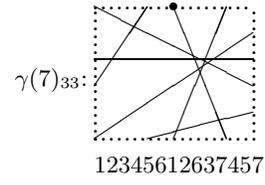
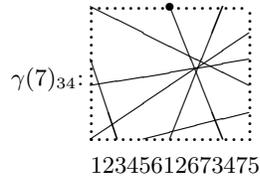
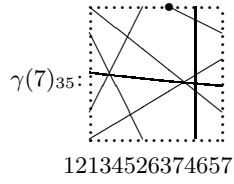
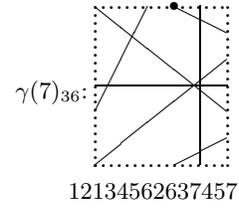
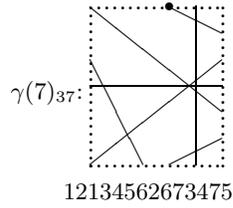
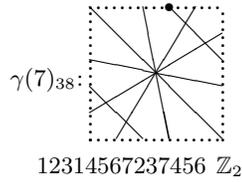
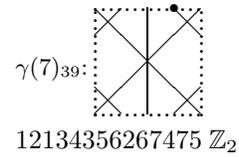
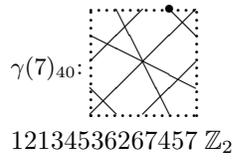
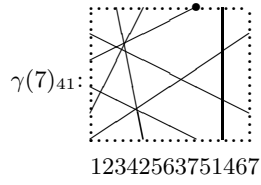
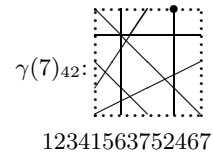
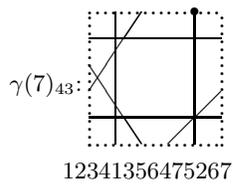
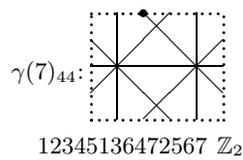
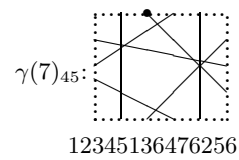
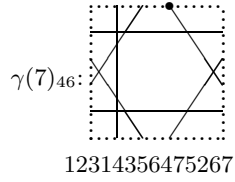
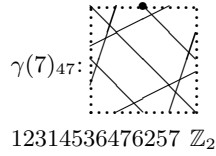
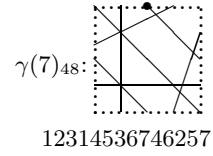
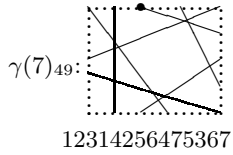
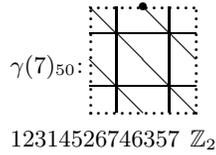
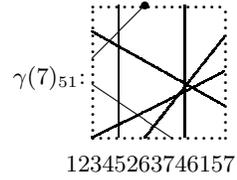
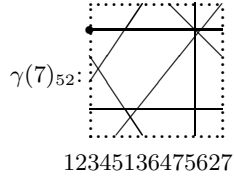
\[%
\]

\section{Appendix 4. Schemas of six-edge glueings.}

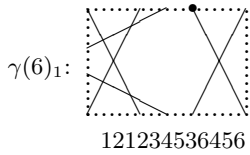
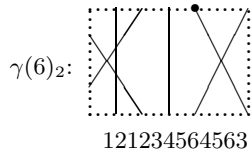
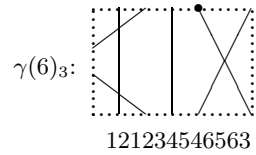
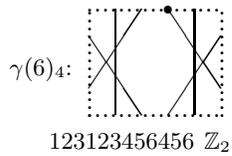
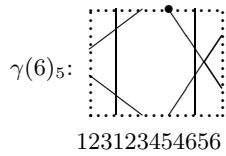
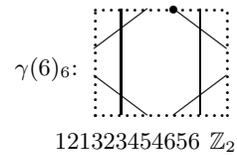
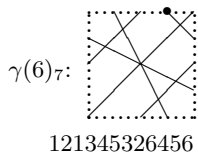
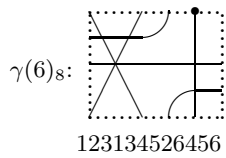
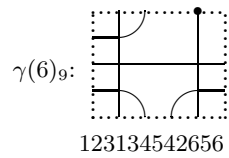
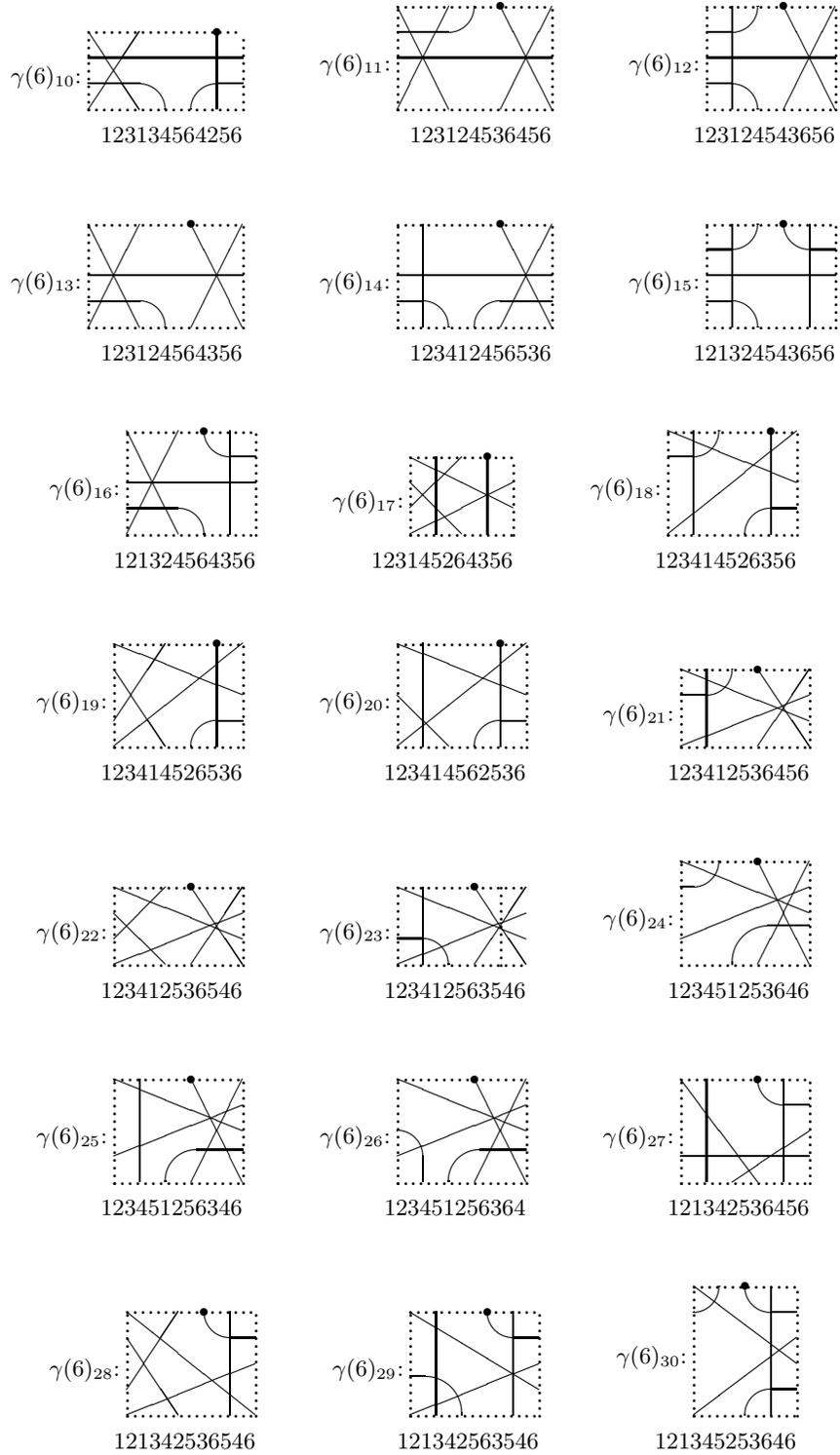
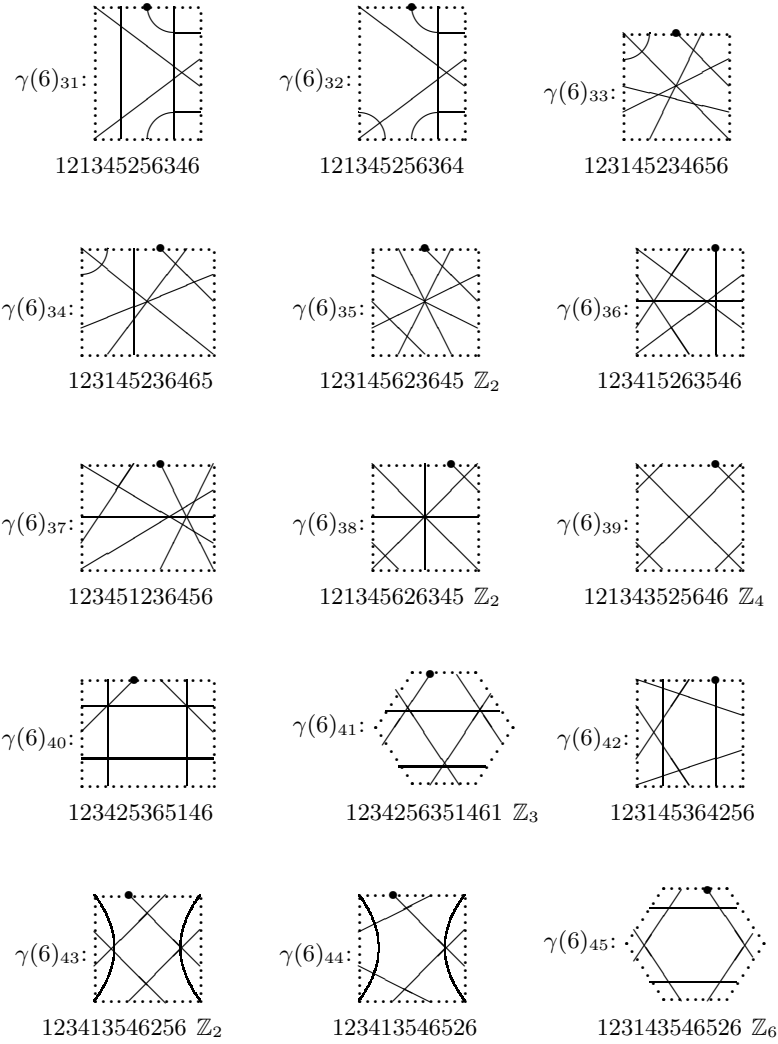
\[%
\]

\section{Appendix 5. Schemas of five-edge glueings.}

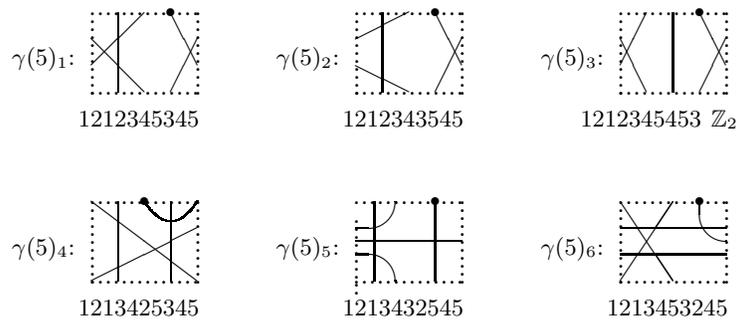
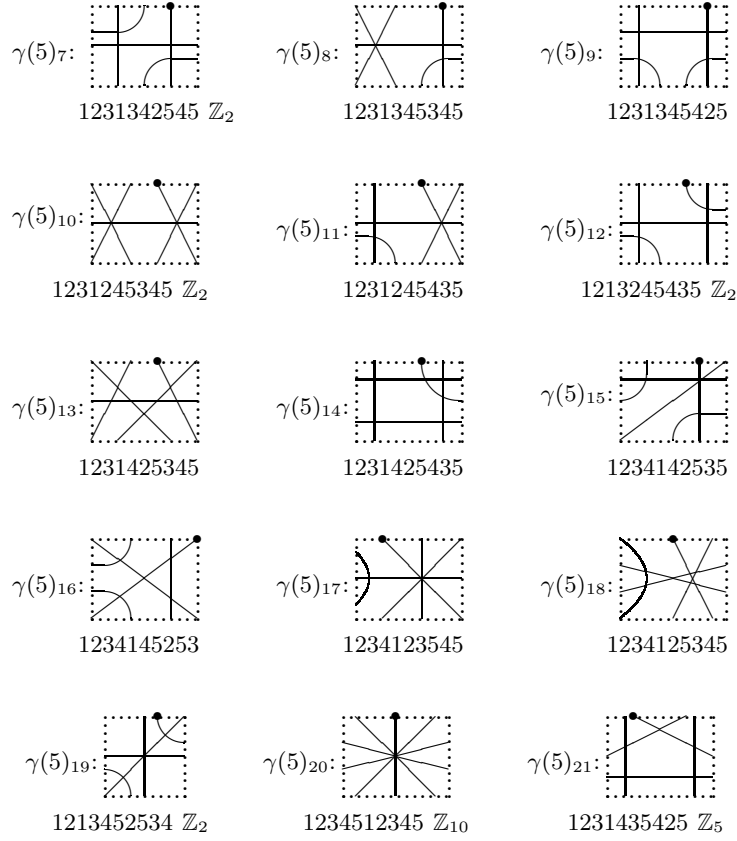
\[%
\]

\section{Appendix 6. Schemas of four-edge glueings}

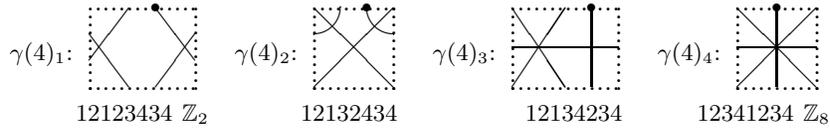
\[%
\]

\section{Appendix 7. Transposed matrix of the map $\partial_8: C_8\to C_7$}

$$%
$$
\par\noindent
\emph{Remark.} To save the space by $\w 1$ we denote the number $-1$. In what
follows we will use the same notation for negative matrix elements.

\section{Appendix 8. Transposed matrix of the map
 $\partial_7: C_7\to C_6$}

$$%
$$

\section{Appendix 8. Matrix of the map $\partial_6: C_6\to C_5$}

$$%
$$

\section{Appendix 9. Matrix of the map $\partial_5: C_5\to C_4$}

$$%
$$
\par\medskip\noindent
Here in the row next to the last in position (20,37) instead of "$a$" must be
$-10$.

\section{Appendix 10. Matrix of the map $\partial_5: C_4\to C_3$}

$$%
$$


\begin{thebibliography}{4}
    \bibitem{Lando} Lando S., Zvonkin A., Graphs on surfaces and their applications.
    Springer-Verlag, 2004.
    \bibitem{Kontz} Kontzevich M., Intersection theory on the moduli space of curves and
    matrix Airy function, Comm. Math. Phys. (1992) v. 147, no. 1, 1--23.
    \bibitem{Koch} Kochetkov Yu., Moduli spaces $\mathcal{M}_{2,1}$ and
    $\mathcal{M}_{3,1}$, Funct. Anal. Appl., 2010, 44(2), 118-124.
    \bibitem{Harer} Harer J., Zagier D., The Euler characteristic of the moduli space of
    curves, Invent. Math., 1986, 85, 457-485.
\end{thebibliography}
\end{document}